\documentclass[12pt]{article}
\usepackage{amsfonts}
\usepackage{amsmath}
\usepackage{amsthm}
\usepackage{cite}
\usepackage{geometry}
\newcommand{\bb}[1]{\mathbb{#1}}

\newtheorem{lem}{Lemma}
\newtheorem{cor}{Corollary}
\newtheorem{thm}{Theorem}
\newtheorem{conj}{Conjecture}
\newtheorem{prop}{Proposition}

\title{Fibonacci-like growth of numerical semigroups of a given genus}
\author{Alex Zhai} 
\date{}

\begin{document}

\maketitle

\begin{abstract}
  \noindent We give an asymptotic estimate of the number of numerical
  semigroups of a given genus. In particular, if $n_g$ is the number
  of numerical semigroups of genus $g$, we prove that

  \[ \lim_{g \rightarrow \infty} n_g \varphi^{-g} = S \]
  
  \noindent where $S$ is a constant, resolving several related
  conjectures concerning the growth of $n_g$. In addition, we show that
  the proportion of numerical semigroups of genus $g$ satisfying $f <
  3m$ approaches $1$ as $g \rightarrow \infty$, where $m$ is the
  multiplicity and $f$ is the Frobenius number.
\end{abstract}

\section{Introduction}

A \emph{numerical semigroup} is defined to be a cofinite subsemigroup
of the non-negative integers. A numerical semigroup $\Lambda$ is said
to have \emph{genus} $g = g(\Lambda)$ if $| \bb{N} \setminus \Lambda |
= g$. We also define the \emph{multiplicity} $m = m(\Lambda) =
\min(\Lambda \setminus \{ 0 \})$. Finally, the \emph{Frobenius number}
$f = f(\Lambda)$ is defined to be $\max(\bb{N} \setminus \Lambda)$.

Let $n_g$ be the number of numerical semigroups of genus $g$. It was
observed by Bras-Amor\'os \cite{BrasAmoros0} that $n_g$ exhibits
Fibonacci-like growth; in particular, it was conjectured that $\lim_{g
  \rightarrow \infty} \frac{n_{g+1}}{n_g} = \varphi$, where $\varphi =
\frac{1 + \sqrt{5}}{2}$ is the golden ratio. Work towards the
resolution of this conjecture has been the subject of a number of
recent papers. It is known that $n_g = \Omega(\varphi^g)$ (see
\cite{BrasAmoros1}, \cite{Zhao}), but the only upper bounds for $n_g$
that have been given are no better than $O((2 - \epsilon)^g)$ (see
\cite{Zhao}, which summarizes the results of \cite{BrasAmoros1} and
\cite{Elizalde}).

Zhao \cite{Zhao} made the observation that most of the numerical
semigroups counted in $n_g$ seem to be of a certain type. Letting
$t_g$ denote the number of numerical semigroups $\Lambda$ of genus $g$
satisfying $f(\Lambda) < 3m(\Lambda)$, Zhao conjectured the following.

\begin{conj} \label{t_g/n_g conjecture}
\[ \lim_{g \rightarrow \infty} \frac{t_g}{n_g} = 1. \]
\end{conj}

\noindent He also gave a formula

\[ \lim_{g \rightarrow \infty} t_g \varphi^{-g} = S, \]

\noindent where $S$ is the value of an infinite sum. It was not
determined whether this sum converges, leaving open the possibility
that $S = \infty$. However, Zhao conjectured that this is not the
case.\footnote{The actual conjecture was phrased in a slightly
  different but equivalent way.}

\begin{conj} \label{t_g convergence conjecture}
\[ \sup_{g \in \bb{N}} t_g \varphi^{-g} < \infty. \]
\end{conj}

Given that Conjecture \ref{t_g convergence conjecture} is true, it
follows that $S$ is finite. Then, if Conjecture \ref{t_g/n_g
  conjecture} holds, it follows that $\lim_{g \rightarrow \infty} n_g
\varphi^{-g} = S$, proving that $n_g$ has Fibonacci-like growth. In
this paper, we prove Conjectures \ref{t_g/n_g conjecture} and \ref{t_g
  convergence conjecture}, thus establishing that Bras-Amor\'os'
original conjecture is correct.\footnote{We do not give an explicit
  estimate of the supremum in Conjecture \ref{t_g convergence
    conjecture}, although such an estimate would be theoretically
  possible. Zhao showed in \cite{Zhao} that it is at least $3.78$,
  with numerical evidence suggesting that it is not much larger than
  that.}  It immediately follows that $\lim_{g \rightarrow \infty}
\frac{n_{g - 1} + n_{g - 2}}{n_g} = 1$, another conjecture of
Bras-Amor\'os \cite{BrasAmoros0}, and it follows that $n_{g + 1} \ge
n_g$ for sufficiently large $g$, verifying a conjecture of Kaplan
\cite{Kaplan} for all but finitely many cases.

Since the set of numerical semigroups of a given genus does not seem
to have much structure, it is difficult to understand exact values of
$n_g$. In order to get around this difficulty, our approach is to use
combinatorial arguments to get Fibonacci-like relations on $n_g$ with
some ``error terms,'' which we then bound. This general idea was
already suggested by Bras-Amor\'os and Bulygin
\cite{BrasAmoros2}. Bras-Amor\'os defines the semigroup tree
\cite{BrasAmoros1}, a combinatorial object that allows us to obtain
the main Fibonacci term, which happens to correspond to $t_g$. Various
bounding techniques are then used to complete the proof.

The rest of the paper is divided into several sections. Section 2
provides an introduction to the semigroup tree and provides much of
the general setup for our approach. Section 3 gives a proof of the
main result under the assumption of a technical lemma. Sections 4 and
5 fill in the details for proving the technical lemma, and Section 6
contains conclusions and further questions.

\section{The semigroup tree}

The semigroup tree is defined in terms of the minimal generators of
numerical semigroups. It is well known that any numerical semigroup
$\Lambda$ has a \emph{minimal generating set} $G$, in the sense that
any set that generates $\Lambda$ contains $G$ (see Theorem 2.7 of
\cite{RosalesSanchez}). The elements of $G$ are called \emph{minimal
  generators}, and it is evident that no minimal generator can be
expressed as the sum of other non-zero elements of the semigroup.

Define an \emph{effective generator} of $\Lambda$ to be a minimal
generator larger than $f(\Lambda)$. This definition was given in
\cite{BrasAmoros2}, and here we further define the \emph{efficacy} $h$
of a semigroup to be the number of effective generators it has. As we
will later see, the effective generators correspond to deviations from
Fibonacci-like growth.

\subsection{Defining the semigroup tree}

We are now in a position to define the semigroup tree, which was first
defined in \cite{BrasAmoros1}. Note that by removing any minimal
generator from a numerical semigroup, we are left with another
numerical semigroup with one higher genus. The idea of the semigroup
tree is to characterize all numerical semigroups as a sequence of such
removals from the semigroup of genus $0$, always removing elements in
increasing order.

We say that a semigroup $\Lambda'$ \emph{descends} from a semigroup
$\Lambda$ if $\Lambda' = \Lambda - \{ \lambda \}$, where $\lambda$ is
an effective generator of $\Lambda$. Clearly, $g(\Lambda') =
g(\Lambda) + 1$. Then, we can consider an infinite tree whose vertices
are numerical semigroups, whose root is the semigroup of genus $0$,
and whose edges are between those pairs of semigroups in which one
descends from the other. It can be shown that each numerical semigroup
appears in this tree exactly once, and furthermore, it appears at
depth $g$ if $g$ is its genus (see \cite{BrasAmoros1} for a more
detailed discussion).

\subsection{Types of descent}

Suppose that $\Lambda' = \Lambda - \{ \lambda \}$ is a numerical
semigroup descending from $\Lambda$. We say that this descent is
\emph{weak} if each effective generator of $\Lambda'$ is also an
effective generator of $\Lambda$. In other words, no ``new'' effective
generator is created. We say the descent is \emph{strong}
otherwise. If the descent is strong, then we say that $\Lambda'$ is a
\emph{strongly descended} numerical semigroup. It will be convenient
to also consider the genus $0$ semigroup to be strongly descended. If
a numerical semigroup $\Lambda''$ is obtained from $\Lambda$ by a
series of weak descents (and no strong descents), then we say that
$\Lambda''$ is a \emph{weak descendent} of $\Lambda$. We will use the
convention that $\Lambda$ is a weak descendent of itself. These
notions are adapted from \cite{BrasAmoros2}.

Now, for a numerical semigroup $\Lambda$, let $N_g(\Lambda)$ denote
the number of weak descendents of $\Lambda$ having genus $g$. Each
numerical semigroup is the weak descendent of a unique strongly
descended ancestor (namely, its nearest strongly descended
ancestor). Thus, if $\mathcal{S}$ is the set of strongly descended
semigroups, then

\[ n_g = \sum_{\Lambda \in \mathcal{S}} N_g(\Lambda). \]

In order to bound this sum, it will be useful to make note of two
lemmas. The first is an observation from \cite{BrasAmoros1} giving a
condition for a numerical semigroup to be strongly descended.

\begin{lem}[Bras-Amor\'os] \label{Bras-Amoros}
A numerical semigroup $\Lambda$ is strongly descended if and only if
$f(\Lambda) + m(\Lambda)$ is a minimal (hence effective) generator of
$\Lambda$.
\end{lem}
\begin{proof}
  See \cite{BrasAmoros1}.
\end{proof}

\noindent The second lemma is an upper bound on $N_g(\Lambda)$.

\begin{lem} \label{N_g bound} For any numerical semigroup $\Lambda$,
  we have $N_g(\Lambda) \le \binom{h(\Lambda)}{g - g(\Lambda)}$ and
  $N_g(\Lambda) \le \varphi^{g - g(\Lambda) + h(\Lambda)}$.
\end{lem}
\begin{proof}
  Let $F_n$ denote the $n$th Fibonacci number. Since $\binom{a}{b} \le
  F_{a + b} \le \varphi^{a + b}$, the second inequality is a
  consequence of the first. To show the first inequality, note that
  each weak descendent of $\Lambda$ is obtained by removing several
  effective generators of $\Lambda$. If the genus of the weak
  descendent is $g$, then $g - g(\Lambda)$ effective generators must
  be removed. There are at most $\binom{h(\Lambda)}{g - g(\Lambda)}$
  ways to choose $g - g(\Lambda)$ effective generators to remove;
  consequently, $N_g(\Lambda) \le \binom{h(\Lambda)}{g - g(\Lambda)}$.
\end{proof}

These lemmas give a rough idea of the general approach---use the
condition given in Lemma \ref{Bras-Amoros} to bound the number of
strongly descended numerical semigroups $\Lambda$, and then use Lemma
\ref{N_g bound} to bound $N_g(\Lambda)$. This will be carried out in
detail in the sections that follow.

\section{The main result}

Recall that the main result of this paper is

\begin{thm} \label{main result}
Let $n_g$ be the number of numerical semigroups of genus $g$. Then,

\[ \lim_{g \rightarrow \infty} \frac{n_g}{\varphi^g} = S, \]

where $S$ is a constant.
\end{thm}

\noindent In this section, we will prove the main result under the assumption
that the following inequality holds:

\begin{lem} \label{main lemma}
Let $\mathcal{S}(m, f)$ be the set of all strongly descended numerical
semigroups having multiplicity $m$ and Frobenius number $f$. Then,

\[ \sum_{\Lambda \in \mathcal{S}(m, f)} \varphi^{-g(\Lambda) +
  h(\Lambda)} \le 5 (f - m) \left( \frac{1.618}{\varphi} \right)^{f -
  m - 1}. \]
\end{lem}

\noindent Establishing this inequality is actually a key step in
showing Theorem \ref{main result}, but the proof of the inequality is
technical and is therefore deferred to the end of the paper.

From the previous section, our task is to estimate 

\[ n_g = \sum_{\Lambda \in \mathcal{S}} N_g(\Lambda). \]

\noindent We do this by partitioning $\mathcal{S}$ into three subsets
and summing over the three parts separately. Let $\mathcal{S}_1$
denote the set of strongly descended semigroups $\Lambda$ such that
$h(\Lambda) + g(\Lambda) < g$. Let $\mathcal{S}_2$ denote the set of
strongly descended semigroups $\Lambda$ such that $h(\Lambda) +
g(\Lambda) \ge g$ and $g(\Lambda) - h(\Lambda) <
\frac{g}{3}$. Finally, let $\mathcal{S}_3$ denote the set of strongly
descended semigroups $\Lambda$ such that $h(\Lambda) + g(\Lambda) \ge
g$ and $g(\Lambda) - h(\Lambda) \ge \frac{g}{3}$.

It is evident that the $\mathcal{S}_i$ partition $\mathcal{S}$. Thus,
we can write $n_g = n_{g, 1} + n_{g, 2} + n_{g, 3}$, where

\[ n_{g, i} = \sum_{\Lambda \in \mathcal{S}_i} N_g(\Lambda). \]

\noindent Note that if $\Lambda \in \mathcal{S}_1$, then by Lemma
\ref{N_g bound}, $N_g(\Lambda) = 0$ because $h(\Lambda) < g -
g(\Lambda)$. It follows that $n_{g, 1} = 0$. In the next two
subsections, we estimate $n_{g, 2}$ and $n_{g, 3}$.

\subsection{Estimating $n_{g, 2}$}

We will show that $n_{g, 2} = O(\varphi^g)$ and $n_{g, 2} \le
t_g$. The relevant properties of semigroups $\Lambda$ in
$\mathcal{S}_2$ are that $\Lambda$ is strongly descended and
$2h(\Lambda) > g(\Lambda)$. The first property is immediate from the
definition $\mathcal{S}$, and the second property follows from
manipulating the inequalities defining $\mathcal{S}_2$:

\[ 3(g(\Lambda) - h(\Lambda)) < g \le g(\Lambda) + h(\Lambda) \]
\[ g(\Lambda) < 2h(\Lambda). \]

\noindent We will define any semigroup satisfying these two properties
to be \emph{orderly}, and rather than work with semigroups in
$\mathcal{S}_2$ directly, it will be more convenient to make
observations about orderly semigroups in general and apply them to
$\mathcal{S}_2$. These observations stem from the following
proposition:

\begin{prop} \label{f < 2m}
If $\Lambda$ is orderly, then $f(\Lambda) < 2m(\Lambda)$.
\end{prop}
\begin{proof}
First, we observe that for any numerical semigroup $\Lambda$, the
effective generators must lie in the interval $[f(\Lambda) + 1,
  f(\Lambda) + m(\Lambda)]$. Thus, $h(\Lambda) \le m(\Lambda)$.

Since $\Lambda$ is strongly descended, we know that $f(\Lambda) +
m(\Lambda)$ is an effective generator. Consequently, $\Lambda$
contains at most half of the integers in the interval $[m(\Lambda),
  f(\Lambda)]$, since no two elements of $\Lambda$ can sum to
$f(\Lambda) + m(\Lambda)$.\footnote{Note that $m(\Lambda) \le
  f(\Lambda) + 1$ for all numerical semigroups $\Lambda$. In the case
  that $m(\Lambda) = f(\Lambda) + 1$, we take $[m(\Lambda),
    f(\Lambda)]$ to be empty, and our analysis still carries through.}
This already forces at least $\frac{f(\Lambda) - m(\Lambda) + 1}{2}$
elements of the interval $[m(\Lambda), f(\Lambda)]$ to be absent from
$\Lambda$, so

\[ m(\Lambda) - 1 + \frac{f(\Lambda) - m(\Lambda) + 1}{2} \le g(\Lambda) \le 2h(\Lambda) - 1 \le 2m(\Lambda) - 1. \]

\noindent Rearranging yields $f(\Lambda) \le 3m(\Lambda) - 1$.

Note that the number of elements of $\Lambda$ in $[m(\Lambda),
  f(\Lambda)]$ is $f(\Lambda) - g(\Lambda)$. Since $f(\Lambda) \le
3m(\Lambda) - 1$, the intervals $[m(\Lambda), 2m(\Lambda) - 1]$ and
$[f(\Lambda) - m(\Lambda) + 1, f(\Lambda)]$ cover $[m(\Lambda),
  f(\Lambda)]$, so at least one of those intervals has at least half
of the elements of $\Lambda$ in $[m(\Lambda), f(\Lambda)]$. In other
words, one of the intervals contains at least $\frac{f(\Lambda) -
  g(\Lambda)}{2}$ elements of $\Lambda$.

Therefore, there are at least $\frac{f(\Lambda) - g(\Lambda)}{2}$
residues $r$ modulo $m(\Lambda)$ for which there exists $\lambda \in
\Lambda$ with $\lambda \le f(\Lambda)$ and $\lambda \equiv r \bmod
m(\Lambda)$. If such a $\lambda$ exists, it is impossible for
$\Lambda$ to have an effective generator congruent to $r$ modulo
$m(\Lambda)$, since such a generator would be the sum of $\lambda$ and
a multiple of $m(\Lambda)$. Consequently, $\Lambda$ has at most
$m(\Lambda) - \frac{f(\Lambda) - g(\Lambda)}{2}$ effective generators.

We thus have

\[ g(\Lambda) < 2h(\Lambda) \le 2m(\Lambda) - \left( f(\Lambda) - g(\Lambda) \right) \]
\[ f(\Lambda) < 2m(\Lambda), \]

\noindent as desired.
\end{proof}

\begin{cor} \label{m >= f + h - g}
If $\Lambda$ is an orderly semigroup, then $m(\Lambda) \ge f(\Lambda)
+ h(\Lambda) - g(\Lambda)$.
\end{cor}
\begin{proof}
By Proposition \ref{f < 2m}, for any $\lambda \in \Lambda \cap
[m(\Lambda), f(\Lambda)]$, we know that $\lambda + m(\Lambda) \in
[f(\Lambda) + 1, f(\Lambda) + m]$, and $\lambda + m(\Lambda)$ cannot
be an effective generator.

Note that there are $f(\Lambda) - g(\Lambda)$ elements of $\Lambda$ in
$[m(\Lambda), f(\Lambda)]$. Hence, $\Lambda$ has at most $m(\Lambda) -
\left( f(\Lambda) - g(\Lambda) \right)$ effective generators. This
gives us the inequality 

\[ h(\Lambda) \le m(\Lambda) - f(\Lambda) + g(\Lambda), \]

\noindent which is the desired inequality upon rearranging terms.
\end{proof}

\begin{cor}
If $\Lambda'$ is a weak descendent of an orderly semigroup, then
$f(\Lambda') < 3m(\Lambda')$.
\end{cor}
\begin{proof}
Let $\Lambda$ be the orderly semigroup for which $\Lambda'$ is the
weak descendent of $\Lambda$. By Lemma \ref{Bras-Amoros}, the largest
effective generator of $\Lambda$ is $f(\Lambda) + m(\Lambda)$. Since
$\Lambda'$ is obtained from $\Lambda$ by removing effective generators
of $\Lambda$, it follows that $f(\Lambda') \le f(\Lambda) +
m(\Lambda)$. Meanwhile, $m(\Lambda') = m(\Lambda)$, so

\[ f(\Lambda') \le f(\Lambda) + m(\Lambda) < 3m(\Lambda) = 3m(\Lambda'). \]
\end{proof}

\begin{cor}
$n_{g, 2} \le t_g$.
\end{cor}
\begin{proof}
By definition, $n_{g, 2}$ counts the number of weak genus $g$
descendents of elements of $\mathcal{S}_2$. Since all elements of
$\mathcal{S}_2$ are orderly, all weak descendents of elements of
$\mathcal{S}_2$ are counted under $t_g$ by the previous
corollary. Thus, $n_{g, 2} \le t_g$.
\end{proof}

\noindent We next define the function $\tau(\Lambda, \Delta) = \{ 0 \}
\cup \left( \left( \Lambda \setminus \{ 0 \} \right) + \Delta \right)$
for a numerical semigroup $\Lambda$ and $\Delta \in \bb{Z}$. In
essence, $\tau(\Lambda, \Delta)$ is a shift of the non-zero elements
of $\Lambda$ by $\Delta$. We record several basic properties of $\tau$
as lemmas.

\begin{lem} \label{tau basic properties}
Let $\Lambda$ be a numerical semigroup, and suppose that $\Lambda' =
\tau(\Lambda, \Delta)$ is also a numerical semigroup. Then,
$f(\Lambda') = f(\Lambda) + \Delta$, $m(\Lambda') = m(\Lambda) +
\Delta$, and $g(\Lambda') = g(\Lambda) + \Delta$.
\end{lem}
\begin{proof}
These are all immediate from the definition of $\tau$.
\end{proof}

\begin{lem} \label{tau semigroup condition}
If a numerical semigroup $\Lambda$ and an integer $\Delta$ satisfy
$f(\Lambda) < 2m(\Lambda) + \Delta$, then $\tau(\Lambda, \Delta)$ is
also a numerical semigroup.
\end{lem}
\begin{proof}
Let $\Lambda' = \tau(\Lambda, \Delta)$. Note that $\min \left( \Lambda
\setminus \{ 0 \} \right) + \Delta = m(\Lambda) + \Delta \ge
f(\Lambda) - m(\Lambda) + 1 \ge 0$. Hence, all elements of $\Lambda'$
are non-negative, and it is easy to see that $\max \left( \bb{N}
\setminus \Lambda' \right) = \max \left( \bb{N} \setminus \Lambda
\right) + \Delta = f(\Lambda) + \Delta$.

For any non-zero $\lambda_1, \lambda_2 \in \Lambda'$, we have

\[ \lambda_1 + \lambda_2 \ge 2 \min \left( \Lambda' \setminus \{ 0 \}
\right) \]
\[ = 2 m(\Lambda) + 2 \Delta \]
\[ > f(\Lambda) + \Delta \]
\[ = \max \left( \bb{N} \setminus \Lambda' \right). \]

\noindent Hence, $\Lambda'$ is closed under addition, so it is a
numerical semigroup.
\end{proof}

\begin{lem} \label{non-effective bijection}
For a numerical semigroup $\Lambda$, let $L = L(\Lambda) = \{ x \in
[0, f(\Lambda) - m(\Lambda)] \mid m(\Lambda) + x \in \Lambda \}$. If
$f(\Lambda) < 2m(\Lambda)$, then $\lambda \in [f(\Lambda) + 1,
  f(\Lambda) + m]$ is an effective generator if and only if $\lambda -
2m(\Lambda) \not \in L + L$.
\end{lem}
\begin{proof}
By definition, $\lambda$ is an effective generator if and only if
there do not exist two non-zero elements $\lambda_1, \lambda_2 \in
\Lambda$ such that $\lambda = \lambda_1 + \lambda_2$. Since $\lambda
\le f(\Lambda) + m(\Lambda)$, and $\lambda_1, \lambda_2 \ge
m(\Lambda)$, it follows that we need only consider the situation where
$\lambda_1, \lambda_2 \le f(\Lambda)$.

In other words, we are only concerned with the case $\lambda_1,
\lambda_2 \in L + m(\Lambda)$, so $\lambda$ is an effective generator
if and only if $\lambda \not \in \left( L + m(\Lambda) \right) +
\left( L + m(\Lambda) \right)$. Subtracting $2m(\Lambda)$ gives the
result.
\end{proof}

\begin{cor} \label{non-effective corollary}
Let $\Lambda$ be a numerical semigroup, and suppose that $\Lambda' =
\tau(\Lambda, \Delta)$ is also a numerical semigroup. Suppose further
that $f(\Lambda) < 2m(\Lambda)$ and $f(\Lambda') <
2m(\Lambda')$. Then, $\Lambda$ is strongly descended if and only if
$\Lambda'$ is, and $m(\Lambda) - h(\Lambda) = m(\Lambda') -
h(\Lambda')$.
\end{cor}
\begin{proof}
Using the notation of Lemma \ref{non-effective bijection}, note that
$L(\Lambda) = L(\Lambda')$, so let us use $L$ to denote simultaneously
$L(\Lambda)$ and $L(\Lambda')$. Let $K$ denote the set of numbers in
$[f(\Lambda) + 1, f(\Lambda) + m(\Lambda)]$ that are not effective
generators of $\Lambda$, and similarly, let $K'$ denote the set of
numbers in $[f(\Lambda') + 1, f(\Lambda') + m(\Lambda')]$ that are not
effective generators of $\Lambda'$.

Suppose $\lambda$ is an element of $K$. Then, by Lemma
\ref{non-effective bijection}, $\lambda - 2m(\Lambda) \in L + L$. This
implies first of all that $\lambda \ge 2m(\Lambda)$, and so by Lemma
\ref{tau basic properties}, $\lambda + 2 \Delta \ge 2m(\Lambda') \ge
f(\Lambda') + 1$. Also by Lemma \ref{tau basic properties}, we know
from $\lambda \le f(\Lambda) + m(\Lambda)$ that $\lambda + 2 \Delta
\le f(\Lambda') + m(\Lambda')$.

Thus, $\lambda + 2 \Delta \in [f(\Lambda') + 1, f(\Lambda') +
  m(\Lambda')]$, and $\lambda + 2 \Delta - 2m(\Lambda') = \lambda -
2m(\Lambda) \in L + L$, which implies by Lemma \ref{non-effective
  bijection} that $\lambda + 2 \Delta$ is not an effective generator
of $\Lambda'$, and so $\lambda + 2 \Delta \in K'$.

Consequently, $\lambda \mapsto \lambda + 2 \Delta$ gives an injection
of $K$ into $K'$. Since all of the arguments above still hold when the
roles of $\Lambda$ and $\Lambda'$ are reversed, we find that this
injection is in fact a bijection.

By Lemma \ref{Bras-Amoros}, $\Lambda$ is strongly descended if and
only if $f(\Lambda) + m(\Lambda) \not \in K$, which occurs if and only
if $f(\Lambda) + m(\Lambda) + 2 \Delta = f(\Lambda') + m(\Lambda')
\not \in K'$. This in turn is equivalent to $\Lambda'$ being strongly
descended, proving the first claim of the corollary. The second claim
follows upon noting that $m(\Lambda) - h(\Lambda) = |K| = |K'| =
m(\Lambda') - h(\Lambda')$.
\end{proof}

We next prove two results having to do with counting the number of
certain semigroups. Let $M(g, h)$ denote the set of strongly descended
numerical semigroups of genus $g$ having $h$ effective
generators. Then, the following lemma holds.

\begin{lem} 
$|M(g, h)| = |M(2g - 2h + 1, g - h + 1)|$ whenever $g < 2h$.
\end{lem}
\begin{proof}
Let $\Delta = 2h - g - 1$, and note that $\Delta \ge 0$. It suffices
to show that $\Lambda \mapsto \tau(\Lambda, \Delta)$ is a bijection
from $M(2g - 2h + 1, g - h + 1)$ to $M(g, h)$, with inverse given by
$\Lambda \mapsto \tau(\Lambda, -\Delta)$. Note that the semigroups in
$M(g, h)$ and $M(2g - 2h + 1, g - h + 1)$ are orderly.

If $\Lambda \in M(2g - 2h + 1, g - h + 1)$, then $f(\Lambda) <
2m(\Lambda) \le 2m(\Lambda) + \Delta$ by Proposition \ref{f <
  2m}. Thus, by Lemma \ref{tau semigroup condition}, $\Lambda' =
\tau(\Lambda, \Delta)$ is a semigroup. In addition, $f(\Lambda') =
f(\Lambda) + \Delta < 2m(\Lambda) + 2 \Delta = 2m(\Lambda')$, so
Corollary \ref{non-effective corollary} applies. Thus, $\Lambda'$ is
strongly descended, and $h(\Lambda') = m(\Lambda') - m(\Lambda) +
h(\Lambda) = \Delta + (g - h + 1) = h$. By Lemma \ref{tau basic
  properties}, we also know that $g(\Lambda') = g(\Lambda) + \Delta =
g$. Hence, $\Lambda' \in M(g, h)$.

Next, suppose that $\Lambda \in M(g, h)$. Then, by Corollary \ref{m >=
  f + h - g} and the general fact that $f(\Lambda) \ge g(\Lambda)$, we
find that 

\[ f(\Lambda) \le 2f(\Lambda) - g(\Lambda) \]
\[ \le 2m(\Lambda) + g(\Lambda) - 2h(\Lambda) \]
\[ < 2m(\Lambda) - \Delta. \]

\noindent Therefore, Corollary \ref{non-effective corollary} applies,
and so it can be verified that $\Lambda' = \tau(\Lambda, -\Delta)$
belongs to $M(2g - 2h + 1, g - h + 1)$ by the same arguments as
before. We thus conclude that $M(g, h)$ and $M(2g - 2h + 1, g - h +
1)$ are in bijection, proving the lemma.
\end{proof}

\begin{lem}
$\sum_{i = 0}^\infty |M(2i + 1, i + 1)| \varphi^{-i}$ converges.
\end{lem}
\begin{proof}
Let $\Lambda$ be any semigroup in $M(2i + 1, i + 1)$. Each number in
$[f(\Lambda) + 1, 2m(\Lambda) - 1]$ is an effective generator of
$\Lambda$ (since it cannot be the sum of two non-zero elements of
$\Lambda$), so $h(\Lambda) \ge 2m(\Lambda) - f(\Lambda) - 1$. Thus, we
have

\[ f(\Lambda) \ge g(\Lambda) = 2i + 1 = 2(i + 1) - 1 \]
\[ = 2h(\Lambda) - 1 \ge 4m(\Lambda) - 2f(\Lambda) - 3, \]

\noindent which upon rearranging yields $3 + 4(f(\Lambda) -
m(\Lambda)) \ge f(\Lambda)$. We therefore find that

\[ \sum_{i = 0}^\infty |M(2i + 1, i + 1)| \varphi^{-i} = \sum_{i = 0}^\infty \sum_{\Lambda \in M(2i + 1, i + 1)} \varphi^{-g(\Lambda) - h(\Lambda)} \]
\[ \le \sum_{\substack{f, m \\ f \ge m - 1 \\ f \le 3 + 4(f - m)}} \sum_{\Lambda \in \mathcal{S}(m, f)} \varphi^{-g(\Lambda) - h(\Lambda)} \]
\[ \le \varphi \sum_{\substack{f, m \\ f \ge m - 1 \\ f \le 3 + 4(f - m)}} (f - m) \left( \frac{1.618}{\varphi} \right)^{f - m - 1} \]
\[ \le \varphi \sum_{k = 0}^\infty (3 + 4k)k \left( \frac{1.618}{\varphi} \right)^{k - 1} < \infty. \]

\noindent This proves the lemma.
\end{proof}

Having established several results relating to orderly semigroups, we
are ready to estimate $n_{g, 2}$. We have

\[ n_{g, 2} = \sum_{\Lambda \in \mathcal{S}_2} N_g(\Lambda) \]
\[ = \sum_{0 \le i < \frac{g}{3}} \sum_{\substack{\Lambda \in \mathcal{S}_2 \\ g(\Lambda) - h(\Lambda) = i}} N_g(\Lambda) \]
\[ \le \sum_{0 \le i < \frac{g}{3}} \sum_{\substack{\Lambda \in \mathcal{S}_2 \\ g(\Lambda) - h(\Lambda) = i}} \binom{h(\Lambda)}{g - g(\Lambda)} \]
\[ \le \sum_{0 \le i < \frac{g}{3}} \sum_{i < h \le g - i} |M(i + h, h)| \binom{h}{g - i - h} \]
\[ = \sum_{0 \le i < \frac{g}{3}} |M(2i + 1, i + 1)| \binom{h}{g - i - h} \]
\[ = \sum_{0 \le i < \frac{g}{3}} |M(2i + 1, i + 1)| F_{g - i + 1} \]
\[ \le \varphi^g \sum_{0 \le i < \frac{g}{3}} |M(2i + 1, i + 1)| \varphi^{-i}. \]

The sum in the last expression is bounded, since $\sum_{i = 0}^\infty
|M(2i + 1, i + 1)| \varphi^{-i}$ converges. Thus, we find that $n_{g,
  2} \varphi^{-g}$ is bounded.

\subsection{Estimating $n_{g, 3}$}

Consider a numerical semigroup $\Lambda \in \mathcal{S}_3$. We first
claim that $h(\Lambda) \ge 2m(\Lambda) - f(\Lambda) - 1$.

If $f(\Lambda) + 1 \ge 2m(\Lambda)$, the claim holds
trivially. Otherwise, it is easy to check that the numbers in the
interval $[f(\Lambda) + 1, 2m(\Lambda) - 1]$ are all effective
generators. The interval $[f(\Lambda) + 1, 2m(\Lambda) - 1]$ contains
$2m(\Lambda) - f(\Lambda) - 1$ numbers, so the claim is proven.

Since $\Lambda \in \mathcal{S}_3$, we have $g(\Lambda) - h(\Lambda)
\ge \frac{g}{3}$. Combining this with the claim above yields 

\[ g(\Lambda) \ge \frac{g}{3} + 2m(\Lambda) - f(\Lambda) - 1. \]

\noindent Noting that $g(\Lambda) \le f(\Lambda)$, we can rearrange this to
obtain

\[ f(\Lambda) - m(\Lambda) \ge \frac{g}{6} - \frac{1}{2} > \frac{g}{6} - 1. \]

\noindent Also, note that if $N_g(\Lambda) > 0$, then we must have
$m(\Lambda) \le g(\Lambda) + 1 \le g + 1$. Combining these facts with
Lemma \ref{main lemma}, we find that

\[ n_{g, 3} = \sum_{\Lambda \in \mathcal{S}_3} N_g(\Lambda) \]
\[ \le \sum_{\substack{\Lambda \in \mathcal{S}_3 \\ m(\Lambda) \le g +
    1}} \varphi^{g - g(\Lambda) + h(\Lambda)} \]
\[ \le \varphi^g \sum_{k > \frac{g}{6} - 1} \sum_{\substack{f - m = k
    \\ m \le g + 1}} \sum_{\Lambda \in \mathcal{S}(m, f)}
\varphi^{-g(\Lambda) + h(\Lambda)} \]
\[ \le 5 \varphi^g \sum_{k > \frac{g}{6} - 1} \sum_{\substack{f - m =
    k \\ m \le g + 1}} k \left( \frac{1.618}{\varphi} \right)^{k -
  1}. \]
\[ \le 5 \varphi^g (g + 1) \sum_{k > \frac{g}{6}} k \left(
  \frac{1.618}{\varphi} \right)^k \]
\[ = o(\varphi^g). \]

\noindent Since we know that $n_g$ grows at least as fast as
$\varphi^g$, this shows that $n_{g, 3}$ makes a negligible
contribution as $g \rightarrow \infty$.

\subsection{Estimating $n_g$}

Now that we have estimated $n_{g, 2}$ and $n_{g, 3}$ separately, it is
possible to give an estimate of $n_g$. Recall that we showed 

\[ n_{g, 2} \le t_g \]
\[ n_{g, 2} = O(\varphi^g) \]
\[ n_{g, 3} = o(\varphi^g). \]

\noindent The last two bounds give $t_g \le n_g = n_{g, 2} + n_{g, 3}
= O(\varphi^g)$, which proves Conjecture \ref{t_g convergence
  conjecture}. Furthermore, we find that $n_g - t_g = n_{g, 3} + n_{g,
  2} - t_g \le n_{g, 3} = o(\varphi^g)$, and it is known that $n_g \ge
\varphi^g$. This proves $\lim_{g \rightarrow \infty} \frac{t_g}{n_g} =
1$, which is Conjecture \ref{t_g/n_g conjecture}. As noted before,
this implies the main result, at least having assumed Lemma \ref{main
  lemma}. In the next two sections, we set out to prove Lemma
\ref{main lemma}.

\section{Some technical preliminaries}

Before proving Lemma \ref{main lemma}, we need to establish another
inequality not directly involving numerical semigroups. Let $S$ be any
finite set of positive integers, and let $m$, $f$, and $d$ be positive
integers satisfying $d < f$. (For the purposes of this section, these
numbers can be considered to bear no relation to numerical semigroups,
but we will later apply our results to the case where $m$ is the
multiplicity and $f$ is the Frobenius number of a numerical
semigroup.)

We say a subset $U \subset S$ is $(m, f, d)$-\emph{admissible} if no
two elements of $U$ sum to $f + m$, and if $x \in U$ and $x + m \in
S$, then $x + m \in U$. Let $\mathcal{A}_{(m, f, d)}(S)$ denote the
set of all $(m, f, d)$-admissible subsets of $S$. Where there is no
risk of confusion, we will drop the $(m, f, d)$ and simply say that
$U$ is an admissible subset of $S$, and we will write $\mathcal{A}(S)$
for $\mathcal{A}_{(m, f, d)}(S)$.

For an admissible subset $U \subset S$, let $E(U, S)$ denote the set
of all integers $x \in S$ such that $x, x + m \not \in U$, but $x - d
\in U$. Let $E'(U, S)$ denote the set of all integers $x \in U$, such
that $x + m \in U$.\footnote{The astute reader may notice that $E'(U,
  S)$ is actually independent of $S$. However, it is convenient to
  write it in this way for sake of consistency with the notation for
  $E(U, S)$ and as a reminder that $U$ is admissible as a subset of
  $S$.} Define $s(U, S) = |E(U, S)| - |E'(U, S)|$. When it is clear
from context what $S$ is, we will simply write $E(U)$, $E'(U)$, and
$s(U)$.

\noindent Define the $(m, f, d)$-\emph{weight} of a set $S$ to be 

\[ \sum_{U \in \mathcal{A}_{(m, f, d)}(S)} \varphi^{-s(U, S)}. \]

\noindent We will denote it by $w_{(m, f, d)}(S)$, or simply $w(S)$
when it is clear what the values of $m$, $f$, and $d$ are. If $S$ is
empty, we define $w(S)$ to be $1$. The main result of this section is
the following lemma.

\begin{lem} \label{interval weight} Let $m$, $f$, and $d$ be positive
  integers such that $d < f$, and let $S = \{ m + d + 1, m + d + 2,
  \ldots , f - 1 \}$. Then,

  \[ w_{(m, f, d)}(S) \le 1.618^{|S| + d + 2}. \]
\end{lem}

The rest of this section is devoted to proving Lemma \ref{interval
  weight}. To see how Lemma \ref{interval weight} is used to prove
Lemma \ref{main lemma}, the reader may wish to skip ahead to the next
section. For the remainder of the section, let $m$, $f$, and $d$ be
fixed positive integers with $d < f$. We first observe that truncating
a set from below can only decrease its weight. More precisely, for any
set $S$, define $V_k(S) = \{ s \in S \mid s > k \}$. Then, the
following inequality holds.

\begin{lem} \label{subtruncation}
  For any set $S$ of positive integers and any $k$, $w(V_k(S)) \le
  w(S)$.
\end{lem}
\begin{proof}
  First of all, since $V_k(S) \subset S$, it is clear that any
  admissible subset $U$ of $V_k(S)$ is also an admissible subset of
  $S$.

  Clearly, $E(U, V_k(S)) \subset E(U, S)$. The reverse is also true;
  let $x$ be any element of $E(U, S)$. By definition, $x \in S$ and $x
  - d \in U$. Since $U \subset V_k(S)$, it follows that $x > k$, and
  so $x \in V_k(S)$. It then follows that $x \in E(U, V_k(S))$. Hence,
  $E(U, V_k(S)) = E(U, S)$.

  Similarly, it is easy to check that $E'(U, V_k(S)) = E'(U,
  S)$. Thus,

  \[ w(V_k(S)) = \sum_{U \in \mathcal{A}(V_k(S))} \varphi^{-s(U,
    V_k(S))} = \sum_{U \in \mathcal{A}(V_k(S))} \varphi^{-s(U, S)} \]
  \[ \le \sum_{U \in \mathcal{A}(S)} \varphi^{-s(U, S)} = w(S), \]

  \noindent as desired.
\end{proof}

Another important observation is that the weight is submultiplicative
in a certain sense. In particular, we have the following lemma.

\begin{lem} \label{submultiplicativity}
  If $S_1$ and $S_2$ are sets such that $f + m \not \in S_1 + S_2$,
  and furthermore, for any $s_1 \in S_1$ and $s_2 \in S_2$, we have
  $s_1 \not \equiv s_2 \bmod m$, then $w(S_1 \cup S_2) \le w(S_1)
  w(S_2)$.
\end{lem}
\begin{proof}
  Let $S = S_1 \cup S_2$, and note that $S_1$ and $S_2$ are
  disjoint. It is easy to check that $U \subset S$ is an admissible
  subset of $S$ if and only if $U \cap S_1$ and $U \cap S_2$ are
  admissible subsets of $S_1$ and $S_2$, respectively. Thus, there is
  a bijection between admissible subsets $U$ of $S$ and pairs of
  admissible subsets $(U_1, U_2)$ of $S_1$ and $S_2$; it is given by
  $U \mapsto (U \cap S_1, U \cap S_2)$.

  Next, for each admissible subset $U$ of $S$, we claim that

  \[ s(U, S) \ge s(U_1, S_1) + s(U_2, S_2), \]

  \noindent where $U_i = U \cap S_i$. Note that if $x \in E(U_1,
  S_1)$, then $x \in S_1$, but $x, x + m \not \in U_1$ and $x - d \in
  U_1 \subset U$. Since $S_1$ and $S_2$ lie in distinct residue
  classes modulo $m$, we know that $x, x + m \not \in S_2$. It follows
  that $x \in E(U, S)$.

  Hence, $E(U_1, S_1) \subset E(U, S)$, and analogously, $E(U_2, S_2)
  \subset E(U, S)$. Upon observing that $E(U_1, S_1)$ and $E(U_2,
  S_2)$ are disjoint because $S_1$ and $S_2$ are disjoint, this shows
  that $|E(U, S)| \ge |E(U_1, S_1)| + |E(U_2, S_2)|$.

  In addition, again using the fact that $S_1$ and $S_2$ lie in
  distinct residue classes modulo $m$, it is easy to see that $E'(U,
  S) = E'(U_1, S_1) \cup E'(U_2, S_2)$, and $E'(U_1, S_1)$ is disjoint
  from $E'(U_2, S_2)$. Hence, $|E'(U, S)| = |E'(U_1, S_1)| + |E'(U_2,
  S_2)|$.

  Subtracting this identity from the previous inequality yields $s(U,
  S) \ge s(U_1, S_1) + s(U_2, S_2)$. Then, using the bijection between
  admissible subsets of $S$ and pairs of admissible subsets of $S_1$
  and $S_2$, we find that

  \[ w(S) = \sum_{U \in \mathcal{A}(S)} \varphi^{-s(U, S)} =
  \sum_{\substack{U_1 \in \mathcal{A}(S_1) \\ U_2 \in
      \mathcal{A}(S_2)}} \varphi^{-s(U_1 \cup U_2, S)} \le
  \sum_{\substack{U_1 \in \mathcal{A}(S_1) \\ U_2 \in
      \mathcal{A}(S_2)}} \varphi^{-s(U_1, S_1) - s(U_2, S_2)} \]
  \[ = \sum_{U_1 \in \mathcal{A}(S_1)} \varphi^{-s(U_1, S_1)}
  \sum_{U_2 \in \mathcal{A}(S_2)} \varphi^{-s(U_2, S_2)} = w(S_1)
  w(S_2), \]

  \noindent as desired.
\end{proof}

Lemma \ref{submultiplicativity} allows us to bound the weight of a set
by partitioning it and bounding the different parts separately. To
this end, let $r$ be an integer between $0$ and $m - 1$, and define
$S(r)$ to be the set of all integers in the range $[m, f]$ that are
congruent to $r$ or $f - r$ modulo $m$. Let $I(r)$ denote the number
of integers in the interval $[m, f]$ congruent to $r$ modulo $m$. 

In more explicit terms, $S(r)$ is the set $\{ r + m, r + 2m, \ldots ,
r + I(r)m \} \cup \{ f - r, f - r - m, \ldots, f - r - (I(r) - 1)m
\}$. Assuming that $r \not \equiv f - r \bmod m$, any admissible
subset $U$ of $S(r)$ takes the form $\{ r + (i + 1)m, r + (i + 2)m,
\ldots , r + I(r)m \} \cup \{ f - r, f - r - m, \ldots , f - r - (j -
1)m \}$, where $i$ and $j$ are between $0$ and $I(r)$. If $i = I(r)$,
then there are no elements in $U$ congruent to $r$ modulo $m$, and
similarly, if $j = 0$, there are no elements in $U$ congruent to $f -
r$ modulo $m$. In addition, we require that no two elements of an
admissible subset sum to $f + m$, so either $i = I(r)$, $j = 0$, or

\[ (r + (i + 1)m) + (f - r - (j - 1)m) > f + m \]
\[ i \ge j. \]

\noindent We say that $U$ has \emph{signature} $(i, j, I(r))$. The
signature of $U$ completely determines the size of $E'(U, S(r))$, as
the next lemma shows.

\begin{lem} \label{signature to E'}
  Suppose that an admissible subset $U$ of $S(r)$ has signature $(i,
  j, I(r))$. Then,

  \[ |E'(U, S(r))| = \begin{cases}
    j - 1 & \text{if $i = I(r)$ and $1 \le j \le I(r)$}. \\
    I(r) - i - 1 & \text{if $0 \le i \le I(r) - 1$ and $j = 0$}. \\
    I(r) - i + j - 2 & \text{if $I(r) > i \ge j > 0$}. \\
    0 & \text{if $i = I(r)$ and $j = 0$}. \\
    \end{cases} \]
\end{lem}

\begin{proof}
  First, recall that by the preceding discussion, the four cases given
  in the lemma indeed cover all possible signatures of $U$.

  If $U$ has signature $(i, j, I(r))$, then let $A = \{ r + (i + 1)m,
  r + (i + 2)m, \ldots , r + (I(r) - 1)m \}$, and let $B = \{ f - r -
  m, f - r - 2m, \ldots , f - r - (j - 1)m \}$. If $i \ge I(r) - 1$,
  we take $A$ to be empty, and if $j \le 1$, we take $B$ to be
  empty. Then, $E'(U, S(r)) = A \cup B$, and $|E'(U, S(r))| = |A| +
  |B|$.

  Note that $|A| = I(r) - i - 1$ unless $i = I(r)$, in which case $|A|
  = 0$. Similarly, $|B| = j - 1$ unless $j = 0$, in which case $|B| =
  0$. The formula for $|E'(U, S(r))|$ stated in the lemma follows from
  applying the formulas for $|A|$ and $|B|$ to the four cases.
\end{proof}

Let $\ell$ be the integer between $0$ and $m - 1$ congruent to $f - r$
modulo $m$. Define $N(r)$ to be the least non-negative integer such
that $r + nd \ge \ell - nd$. Finally, define $T(r) = \bigcup_{i =
  0}^{N(r) - 1} S(r + id)$ (take $T(r)$ to be empty if $N(r) = 0$).

\subsection{Some bounds on $w(S(r))$ and $w(T(r))$}

Let us first bound $S(r)$ where $r \not \equiv f - r \bmod m$. In the
cases $I(r) = 1$ and $I(r) = 2$, we have the following lemma.

\begin{lem} \label{I(r) = 1 or 2 bounds on S(r)} 
  Let $r$ be an integer satisfying $0 \le r \le m - 1$ and $r \not
  \equiv f - r \bmod m$. Then, the following bounds hold:

  \begin{enumerate}
  \item If $I(r) = 1$, then $|S(r)| = 2$, and
    \begin{enumerate}
    \item $w(S(r)) \le 3$.
    \item $w(S(r) \setminus \{ r + m \}) \le 2 < 0.7726 \cdot
      1.618^{|S(r)|}$.
    \item $w(S(r) \setminus \{ r + m, f - r \}) = 1 < 0.3820 \cdot
      1.618^{|S(r)|}$.
    \end{enumerate}
  \item If $I(r) = 2$, then $|S(r)| = 4$, and
    \begin{enumerate}
    \item $w(S(r)) \le 4 + 2 \varphi < 1.0559 \cdot 1.618^{|S(r)|}$.
    \item $w(S(r) \setminus \{ r + m \}) \le 4 + \varphi < 0.8198 \cdot
      1.618^{|S(r)|}$.
    \item $w(S(r) \setminus \{ r + m, f - m - r \}) \le 4 < 0.5837
      \cdot 1.618^{|S(r)|}$.
    \end{enumerate}
  \end{enumerate}
\end{lem}

\begin{proof} We explicitly determine the possible signatures of
  admissible subsets $U \subset S(r)$. We then bound $w(S(r))$ using
  the inequality

  \[ w(S(r)) = \sum_{U \in \mathcal{A}(S(r))} \varphi^{-s(U)} \le
  \sum_{U \in \mathcal{A}(S(r))} \varphi^{|E'(U, S(r))|}, \]

  \noindent where we can compute $|E'(U, S(r))|$ from the signature of
  $U$ using Lemma \ref{signature to E'}.

  If $I(r) = 1$, then note that $S(r) = \{ r + m, f - r \}$. The
  possible signatures of $U$ are $(0, 0, 1)$, $(1, 0, 1)$, and $(1, 1,
  1)$. In each case, $|E'(U, S(r))| = 0$. Hence, $w(S(r)) \le 3$. It
  is also routine to check that $w(S(r) \setminus \{ r + m \}) \le 2$,
  and $w(S(r) \setminus \{ r + m, f - r\}) = 1$. This proves part (i).

  If $I(r) = 2$, then $S(r) = \{ r + m, r + 2m, f - r, f - r - m
  \}$. The possible signatures of $U$ are $(2, 0, 2)$, $(2, 1, 2)$,
  $(1, 0, 2)$, $(1, 1, 2)$, $(0, 0, 2)$, and $(2, 2, 2)$. In the first
  four cases, $|E'(U, S(r))| = 0$, while in the other two, $|E'(U,
  S(r))| = 1$. Thus, $w(S(r)) \le 4 + 2 \varphi$.

  The admissible subsets of $S(r) \setminus \{ r + m \}$ are the same
  as those of $S(r)$ with the exception of the set having signature
  $(0, 0, 2)$. The admissible subsets of $S(r) \setminus \{ r + m, f -
  r - m \}$ are the same as those of $S(r)$ except the sets having
  signature $(0, 0, 2)$ and $(2, 2, 2)$. Therefore, $w(S(r) \setminus
  \{ r + m \}) \le 4 + \varphi$ and $w(S(r) \setminus \{ r + m, f - r
  - m \}) \le 4$. This proves part (ii), completing the proof.
\end{proof}

Using much the same approach, we can also give estimates of $w(S(r))$
when $I(r) \ge 3$.

\begin{lem} \label{I(r) >= 3 bound on S(r)} Let $r$ be an integer
  satisfying $0 \le r \le m - 1$ and $r \not \equiv f - r \bmod
  m$. Also, suppose that $I(r) \ge 3$. Then $|S(r)| = 2I(r)$, and

  \[ w(S(r)) \le 0.8755 \cdot 1.618^{|S(r)|}. \]
\end{lem}

\begin{proof}
  Let us partition $\mathcal{A}(S(r)) = A_1 \cup A_2 \cup A_3 \cup
  A_4$ according to the four cases of Lemma \ref{signature to
    E'}. More explicitly,

  \begin{enumerate}
  \item $A_1$ consists of those subsets $U$ having signature $(I(r),
    j, I(r))$ where $1 \le j \le I(r)$.
  \item $A_2$ consists of those subsets $U$ having signature $(i, 0,
    I(r))$ where $0 \le i \le I(r) - 1$.
  \item $A_3$ consists of those subsets $U$ having signature $(i, j,
    I(r))$ where $I(r) > i \ge j > 0$.
  \item $A_4$ consists of the single subset $U$ having signature
    $(I(r), 0, I(r))$ (namely, the empty set).
  \end{enumerate}

  \noindent By Lemma \ref{signature to E'}, we have

  \[ \sum_{U \in A_1} \varphi^{|E'(U, S(r))|} = \sum_{j = 1}^{I(r)}
  \varphi^{j - 1} = \sum_{k = 0}^{I(r) - 1} \varphi^k. \]

  \[ \sum_{U \in A_2} \varphi^{|E'(U, S(r))|} = \sum_{i = 0}^{I(r) -
    1} \varphi^{I(r) - i - 1} = \sum_{k = 0}^{I(r) - 1} \varphi^k. \]
  
  \[ \sum_{U \in A_3} \varphi^{|E'(U, S(r))|} = \sum_{I(r) > i \ge j >
    0} \varphi^{I(r) - i + j - 2} \]
  \[ = \sum_{k = 0}^{I(r) - 2} (I(r) - k - 1) \varphi^{I(r) - k - 2} =
  \sum_{k = 0}^{I(r) - 2} (k + 1) \varphi^k. \]
  
  \[ \sum_{U \in A_4} \varphi^{|E'(U, S(r))|} = 1, \]

  \noindent It follows that

  \[ w(S(r)) = \sum_{U \in \mathcal{A}(S(r))} \varphi^{-s(U)} \le
  \sum_{U \in \mathcal{A}(S(r))} \varphi^{|E'(U, S(r))|} \]
  \[ = 1 + 2 \varphi^{I(r) - 1} + \sum_{k = 0}^{I(r) - 2} (k + 3)
  \varphi^k. \]
  
  \noindent Let $W_N$ denote this final expression when $I(r) = N$. It
  remains to show that $W_N \le 1.618^{2N}$ when $N \ge 3$. We can
  verify this in the cases $N = 3$ and $N = 4$ by explicit
  computation:

  \[ W_3 = 1 + 2 \varphi^2 + (3 + 4 \varphi) = 4 + 4 \varphi + 2
  \varphi^2 < 0.8755 \cdot 1.618^6 \]
  \[ W_4 = 1 + 2 \varphi^3 + (3 + 4 \varphi + 5 \varphi^2) = 4 + 4
  \varphi + 5 \varphi^2 + 2 \varphi^3 < 0.8755 \cdot 1.618^8. \]

  \noindent For all $N \ge 4$, we also have

  \[ \left( \varphi + \frac{1}{\varphi} \right) W_N > 1 + 2 \varphi^N
  + 2 \varphi^{N - 2} + \sum_{k = 0}^{N - 2} (k + 3) \varphi^{k + 1} +
  \sum_{k = 0}^{N - 3} (k + 4) \varphi^k \]
  \[ \ge 1 + 2 \varphi^N + \sum_{k = 0}^{N - 1} (k + 2) \varphi^k +
  \sum_{k = 0}^{N - 3} (k + 4) \varphi^k \]
  \[ = 1 + 2 \varphi^N + \sum_{k = 0}^{N - 1} (k + 3) \varphi^k +
  \left( \sum_{k = 0}^{N - 3} (k + 3) \varphi^k - \varphi^{N - 2} -
    \varphi^{N - 1} \right) \]
  \[ \ge 1 + 2 \varphi^N + \sum_{k = 0}^{N - 1} (k + 3) \varphi^k +
  \left( (N - 1) \varphi^{N - 4} + N \varphi^{N - 3} - \varphi^{N - 2}
    - \varphi^{N - 1} \right) \]
  \[ = 1 + 2 \varphi^N + \sum_{k = 0}^{N - 1} (k + 3) \varphi^k +
  \left( (N - 1) + N \varphi - \varphi^2 - \varphi^3 \right)
  \varphi^{N - 4} \]
  \[ > 1 + 2 \varphi^N + \sum_{k = 0}^{N - 1} (k + 3) \varphi^k = W_{N
    + 1}. \]

  \noindent Since $\varphi + \frac{1}{\varphi} < 1.618^2$, the lemma
  follows by induction.
\end{proof}

\begin{cor} \label{I(r) >= 3 bound on T(r)} Let $r$ be an integer
  satisfying $0 \le r \le m - 1$ and $r \not \equiv f - r \bmod
  m$. Also, suppose that $I(r) \ge 3$. Then,

  \[ w(T(r)) \le 1.618^{|T(r)|}. \]
\end{cor}
\begin{proof}
  Recall that $T(r) = \bigcup_{i = 0}^{N(r) - 1} S(r + id)$. Thus, by
  Lemma \ref{submultiplicativity} and Lemma \ref{I(r) >= 3 bound on
    S(r)}, we find that

  \[ w(T(r)) \le \prod_{i = 0}^{N(r) - 1} w(S(r + id)) \le \prod_{i =
    0}^{N(r) - 1} 1.618^{|S(r + id)|} = 1.618^{|T(r)|}, \]

  \noindent as desired.
\end{proof}

Finally, we can give a bound in the case where $r \equiv f - r \bmod
m$.

\begin{lem} \label{2r = f mod m bound on S(r)} Let $r$ be an integer
  satisfying $0 \le r \le m - 1$ and $r \equiv f - r \bmod m$ (if such
  an integer exists). Then, $|S(r)| = I(r)$, and

  \[ w(S(r)) \le 1.618^{|S(r)|}. \]
\end{lem}
\begin{proof}
  Let $\ell$ be the remainder when $f$ is divided by $m$. Then, $r =
  \frac{\ell}{2}$ or $r = \frac{m + \ell}{2}$. In either case, $S(r) =
  \{ m + r, 2m + r, \ldots , I(r)m + r \}$, and so $|S(r)| = I(r)$.

  The non-empty admissible subsets $U$ of $S(r)$ take the form $\{ im
  + r, (i + 1)m + r, \ldots , I(r)m + r \}$, where $i \le I(r)$. Since
  no two elements of $U$ sum to $f + m$, we must also have

  \[ 2(im + r) > f + m = (I(r) + 1)m + 2r \]
  \[ i > \frac{I(r) + 1}{2}. \]

  Note that $E'(U, S(r)) = \{ im + r, (i + 1)m + r, \ldots , (I(r) -
  1)m + r \}$, so $s(U, S(r)) \ge -|E'(U, S(r))| = I(r) -
  i$. Accounting for the fact that $s(\emptyset, S(r)) = 0$, it
  follows that

  \[ w(S(r)) = \sum_{U \in \mathcal{A}(S(r))} \varphi^{-s(U, S(r))}
  \le 1 + \sum_{\frac{I(r) + 1}{2} < i \le I(r)} \varphi^{I(r) - i} \]
  \[ = 1 + \sum_{i = 0}^{\lceil \frac{I(r) - 1}{2} \rceil - 1}
  \varphi^i = 1 + \varphi^{\lceil \frac{I(r) - 1}{2} \rceil + 1} -
  \varphi. \]

  \noindent It is routine to check that this quantity is at most
  $1.618^{I(r)}$.
\end{proof}

\subsection{Bounds on $w(T(r))$ when $I(r) = 1$ or $I(r) = 2$}

We now describe another method for bounding the weight of $T(r)$ in
terms of $N(r)$ and $I(r)$. Let $U$ be an admissible subset of $T(r)$,
and let $U_i = U \cap S(r + id)$ for each $i < N(r)$. Note that $I(r +
id) = I(r)$ for each $i < N(r)$.

By Lemma \ref{signature to E'}, it is clear that the number $e'_i(U)$
of elements in $E'(U, T(r))$ that are congruent to $r + id$ or $f - r
- id$ modulo $m$ (in other words, are in $S(r + id)$) depends only on
the signature of $U_i$. A similar statement holds for the number
$e_i(U)$ of elements in $E(U, T(r))$ congruent to $r + (i + 1)d$ or $f
- r - id$ modulo $m$.

\begin{lem} \label{signature to E} The value of $e_i(U)$ depends only
  on the signatures of $U_i$ and $U_{i + 1}$. In particular, if their
  signatures are $(a_i, b_i, I(r))$ and $(a_{i + 1}, b_{i + 1},
  I(r))$, respectively, then

  \[ e_i(U) = \begin{cases}
    I(r) - a_i + \max(b_{i + 1} - b_i - 1, 0) & \text{if $b_i > 0$ and $a_{i + 1} = I(r)$} \\
    \max(a_{i + 1} - a_i - 1, 0) + b_{i + 1} & \text{if $b_i = 0$ and $a_{i + 1} < I(r)$} \\
    \max(a_{i + 1} - a_i - 1, 0) + \max(b_{i + 1} - b_i - 1, 0) & \text{if $b_i > 0$ and $a_{i + 1} < I(r)$} \\
    I(r) - a_i + b_{i + 1}  & \text{if $b_i = 0$ and $a_{i + 1} = I(r)$}. \\
    \end{cases}. \]
\end{lem}
\begin{proof}
  Let $N_1$ and $N_2$ denote respectively the number of elements in
  $E(U, T(r))$ congruent to $r + (i + 1)d$ and $f - r - id$ modulo
  $m$.

  The number $r + (i + 1)d + km$ is in $E(U, T(r))$ if and only if
  $a_i < k$ and either $a_{i + 1} > k$ or $a_{i + 1} = I(r)$. If $a_{i
    + 1} = I(r)$, then $a_i < k \le I(r)$, and so $N_1 = I(r) -
  a_i$. Otherwise, $a_i < k < a_{i + 1}$, and $N_1 = \max(a_{i + 1} -
  a_i - 1, 0)$.

  Similarly, $f - r - id - km$ is in $E(U, T(r))$ if and only if $b_{i
    + 1} > k$ and either $b_i < k$ or $b_i = 0$. If $b_i = 0$, then $0
  \le k < b_{i + 1}$, and so $N_2 = b_{i + 1}$. Otherwise, $b_i < k <
  b_{i + 1}$, and $N_2 = \max(b_{i + 1} - b_i - 1, 0)$.

  Writing out the formula for $N_1 + N_2$ in the various cases yields
  the result.
\end{proof}

It follows that for every pair of signatures $(u, v)$, there is a
number $G(u, v)$ such that if $U_i$ has signature $u$ and $U_{i + 1}$
has signature $v$, then $e_i(U) - e'_i(U) = G(u, v)$. We set aside the
task of actually computing $G(u, v)$ for the moment, but we note that
$G(u, v)$ does not depend explicitly on $i$ or $U$.

Let the signature of $U_i$ be $u_i$ for each $i$. We find that

\[ s(U, T(r)) = |E(U, T(r))| - |E'(U, T(r))| \ge \sum_{i = 0}^{N(r) -
  1} (e_i(U) - e'_i(U)) \]
\[ = \sum_{i = 0}^{N(r) - 2} G(u_i, u_{i + 1}) + e_{N(r) -
  1}(U) - e'_{N(r) - 1}(U) \]
\[ \ge \sum_{i = 0}^{N(r) - 2} G(u_i, u_{i + 1}) - e'_{N(r) -
  1}(U). \]

\noindent It follows, then, that

\[ w(T(r)) = \sum_{U \in \mathcal{A}(T(r))} \varphi^{-s(U, T(r))} \le
\sum_{U \in \mathcal{A}(T(r))} \varphi^{- \sum_{i = 0}^{N(r) - 2}
  G(u_i, u_{i + 1}) + e'_{N(r) - 1}(U)} \]
\[ = \sum_{U \in \mathcal{A}(T(r))} \varphi^{e'_{N(r) - 1}(U)}
\prod_{i = 0}^{N(r) - 2} \varphi^{-G(u_i, u_{i + 1})}. \]

This bound can be expressed in matrix form.

\begin{lem} \label{matrix bound}

  Let $r$ be an integer, and let the possible signatures of
  $T(r) \cap S(r + id)$ be $\{ s_1, s_2, \ldots , s_k \}$. Let ${\bf
    v}$ denote the $k$-dimensional vector whose $j$th entry is the
  value of $\varphi^{e'_{N(r) - 1}(U)}$ when $u_{N(r) - 1} =
  s_j$. Also, let ${\bf 1}$ denote the $k$-dimensional vector all of
  whose entries are $1$. Finally, let $M$ be the $k \times k$ matrix
  whose $ij$ entry is $\varphi^{-G(s_i, s_j)}$. Then,

  \[ w(T(r)) \le {\bf 1}^T M^{N(r) - 1} {\bf v}. \]
\end{lem}

We may apply this specifically to the cases $I(r) = 1$ and $I(r) = 2$.

\begin{lem} \label{I(r) = 1 bound on T(r)}
  If $I(r) = 1$, then $w(T(r)) < 1.1460 \cdot 1.618^{|T(r)|}$.
\end{lem}
\begin{proof}
  Using the same notation as above, the possible signatures of the
  $U_i$ when $U \in \mathcal{A}(T(r))$ and $I(r) = 1$ are $s_1 = (1,
  0, 1)$, $s_2 = (1, 1, 1)$, and $s_3 = (0, 0, 1)$. By Lemmas
  \ref{signature to E'} and \ref{signature to E}, we find that the
  matrix $\left[ \varphi^{-G(s_i, s_j)} \right]$ is

  \[ \left[ \begin{tabular}{ c c c }                           
      $1$ &  $1$ &  $1$ \\
      $\varphi^{-1}$ &  $1$ &  $1$ \\
      $\varphi^{-2}$ &  $1$ &  $\varphi^{-1}$ \\
    \end{tabular} \right]. \]

  \noindent It is routine to check that ${\bf v}$ is

  \[ \left[ \begin{tabular}{ c }
      $1$ \\
      $1$ \\
      $1$ \\
    \end{tabular} \right]. \]

  \noindent Applying Lemma \ref{matrix bound}, we have

  \[ w(T(r)) \le \left[ \begin{tabular}{c c c}
      $1$ & $1$ & $1$ \\
    \end{tabular} \right] \left[ \begin{tabular}{ c c c }
      $1$ &  $1$ &  $1$ \\
      $\varphi^{-1}$ &  $1$ &  $1$ \\
      $\varphi^{-2}$ &  $1$ &  $\varphi^{-1}$ \\
    \end{tabular} \right]^{N(r) - 1} \left[ \begin{tabular}{ c }
      $1$ \\
      $1$ \\
      $1$ \\
    \end{tabular} \right]. \]

  Let $W_n$ denote the right hand side of the above inequality when
  $N(r) = n$. It is not hard to show by induction that $W_n \le 3
  \cdot 1.618^{2n - 2}$ (see Appendix for the computation
  details). This gives

  \[ w(T(r)) \le W_{N(r)} \le 3 \cdot 1.618^{2N(r) - 2} \]
  \[ < 1.1460 \cdot 1.618^{2N(r)} = 1.1460 \cdot 1.618^{|T(r)|}. \]
\end{proof}

\begin{lem} \label{I(r) = 2 bound on T(r)} 
  If $I(r) = 2$, then $w(T(r)) < 1.0559 \cdot 1.618^{|T(r)|}$.
\end{lem}
\begin{proof}
  The possible signatures of the $U_i$ when $U \in \mathcal{A}(T(r))$
  and $I(r) = 2$ are $s_1 = (2, 0, 2)$, $s_2 = (2, 1, 2)$, $s_3 = (1,
  0, 2)$, $s_4 = (1, 1, 2)$, $s_5 = (0, 0, 2)$, and $s_6 = (2, 2,
  2)$. By Lemmas \ref{signature to E'} and \ref{signature to E}, we
  find that the matrix $\left[ \varphi^{-G(s_i, s_j)} \right]$ is

  \[ \left[ \begin{tabular}{ c c c c c c }
      $1$ &  $\varphi^{-1}$ &  $1$ &  $\varphi^{-1}$ &  $1$ &  $\varphi^{-2}$ \\
      $1$ &  $1$ &  $1$ &  $1$ &  $1$ &  $1$ \\
      $\varphi^{-1}$ &  $\varphi^{-2}$ &  $1$ &  $\varphi^{-1}$ &  $1$ &  $\varphi^{-3}$ \\
      $\varphi^{-1}$ &  $\varphi^{-1}$ &  $1$ &  $1$ &  $1$ &  $\varphi^{-1}$ \\
      $\varphi^{-1}$ &  $\varphi^{-2}$ &  $\varphi$ &  $1$ &  $\varphi$ &  $\varphi^{-3}$ \\
      $\varphi$ &  $\varphi$ &  $\varphi$ &  $\varphi$ &  $\varphi$ &  $\varphi$ \\
  \end{tabular} \right]. \]

  \noindent It is routine to check that ${\bf v}$ is

  \[ \left[ \begin{tabular}{ c }
      $1$ \\
      $1$ \\
      $1$ \\
      $1$ \\
      $\varphi$ \\
      $\varphi$ \\
    \end{tabular} \right]. \]

  \noindent Applying Lemma \ref{matrix bound}, we have

  \[ w(T(r)) \le \left[ \begin{tabular}{c c c c c c}
      $1$ & $1$ & $1$ & $1$ & $1$ & $1$ \\
    \end{tabular} \right] \left[ \begin{tabular}{ c c c c c c }
      $1$ &  $\varphi^{-1}$ &  $1$ &  $\varphi^{-1}$ &  $1$ &  $\varphi^{-2}$ \\
      $1$ &  $1$ &  $1$ &  $1$ &  $1$ &  $1$ \\
      $\varphi^{-1}$ &  $\varphi^{-2}$ &  $1$ &  $\varphi^{-1}$ &  $1$ &  $\varphi^{-3}$ \\
      $\varphi^{-1}$ &  $\varphi^{-1}$ &  $1$ &  $1$ &  $1$ &  $\varphi^{-1}$ \\
      $\varphi^{-1}$ &  $\varphi^{-2}$ &  $\varphi$ &  $1$ &  $\varphi$ &  $\varphi^{-3}$ \\
      $\varphi$ &  $\varphi$ &  $\varphi$ &  $\varphi$ &  $\varphi$ &  $\varphi$ \\
    \end{tabular} \right]^{N(r) - 1} \left[ \begin{tabular}{ c }
      $1$ \\
      $1$ \\
      $1$ \\
      $1$ \\
      $\varphi$ \\
      $\varphi$ \\
    \end{tabular} \right]. \]

  Let $W_n$ denote the right hand side of the above inequality when
  $N(r) = n$. This is a closed form for $W_n$, and straightforward
  computations yield $W_n \le (4 + 2\varphi) \cdot 1.618^{4n - 4}$
  (for the details of the calculation, see Appendix). This gives

  \[ w(T(r)) \le W_{N(r)} \le (4 + 2 \varphi) \cdot 1.618^{4N(r) -
    4} \]
  \[ < 1.0559 \cdot 1.618^{4N(r)} = 1.0559 \cdot 1.618^{|T(r)|}. \]
\end{proof}

\subsection{Proof of Lemma \ref{interval weight}.}

We have now developed the necessary tools to prove Lemma \ref{interval
  weight}.

\begin{proof}[Proof of Lemma \ref{interval weight}]
  Let $\ell$ be the remainder when $f$ is divided by $m$. For
  convenience of notation, define $S(r)$ to be the empty set when $r$
  is not an integer. Note that when $0 \le x < \frac{\ell}{2}$, $I(x)
  = I(0)$, and when $\ell + 1 \le x < \frac{m + \ell}{2}$, $I(x) =
  I(0) - 1$.

  Rather than bounding $w(S)$ directly, it will be more convenient to
  bound the weight of a similar set. Define

  \[ S' = \left( \bigcup_{x = 0}^{\lfloor \frac{\ell}{2} \rfloor} S(x)
  \right) \cup \left( \bigcup_{x = \ell + 1}^{\lfloor \frac{m +
        \ell}{2} \rfloor} S(x) \right), \]

  \noindent and define $S'' = V_{m + d}(S')$. Note that

  \[ S' \cup S(\frac{\ell}{2}) \cup S(\frac{m + \ell}{2}) = \{ m, m +
  1, \ldots , f \} \]
  \[ S'' \cup V_{m + d}(S(\frac{\ell}{2})) \cup V_{m + d}(S(\frac{m +
    \ell}{2})) = S \cup \{ f \}. \]

  We claim that $w(S'') \le 1.618^{|S'|}$. To show this, we consider
  four cases according to the value of $I(0)$.

  \begin{enumerate}
  \item[] \textbf{Case $I(0) > 3$.} We have

    \[ S'' = \left( \bigcup_{x = 0}^{\lfloor \frac{\ell}{2} \rfloor}
      V_{m + d}(S(x)) \right) \cup \left( \bigcup_{x = \ell +
        1}^{\lfloor \frac{m + \ell}{2} \rfloor} V_{m + d}(S(x))
    \right) \]

    Note that wherever $x$ appears in the above equation, $I(x) \ge
    3$. Thus, Lemma \ref{I(r) >= 3 bound on S(r)} applies, and

    \[ w(S'') \le \prod_{x = 0}^{\lfloor \frac{\ell}{2} \rfloor }
    w(S(x)) \prod_{x = \ell + 1}^{\lfloor \frac{m + \ell}{2} \rfloor}
    w(S(x)) \]
    \[ \le \prod_{x = 0}^{\lfloor \frac{\ell}{2} \rfloor }
    1.618^{|S(x)|} \prod_{x = \ell + 1}^{\lfloor \frac{m + \ell}{2}
      \rfloor} 1.618^{|S(x)|} = 1.618^{|S'|}. \]

  \item[] \textbf{Case $I(0) = 3$.} If $d \le \ell$, we again write

    \[ S'' = \left( \bigcup_{x = 0}^{\lfloor \frac{\ell}{2} \rfloor}
      V_{m + d}(S(x)) \right) \cup \left( \bigcup_{x = \ell +
        1}^{\lfloor \frac{m + \ell}{2} \rfloor} V_{m + d}(S(x))
    \right), \]

    noting that $I(x) = 2$ when $\ell + 1 \le x \le \lfloor \frac{m +
      \ell}{2} \rfloor$. Applying Lemmas \ref{I(r) >= 3 bound on S(r)}
    and \ref{I(r) = 2 bound on T(r)}, we have

    \[ w(S'') \le \prod_{x = 0}^{\lfloor \frac{\ell}{2} \rfloor}
    w(S(x)) \cdot \prod_{x = \ell + 1}^{\min(\lfloor \frac{m +
        \ell}{2} \rfloor, \ell + d)} w(T(x)) \]
    \[ \le \prod_{x = 0}^{\lfloor \frac{\ell}{2} \rfloor} 0.8557 \cdot
    1.618^{|S(x)|} \cdot \prod_{x = \ell + 1}^{\min(\lfloor \frac{m +
        \ell}{2} \rfloor, \ell + d)} 1.0559 \cdot 1.618^{|T(x)|} \]
    \[ \le 0.8557^{\ell/2} \cdot 1.0559^d \cdot 1.618^{|S'|} \le
    1.618^{|S'|}. \]

    If $\ell < d \le \frac{m + \ell}{2}$, then we use the
    decomposition

    \[ S'' = \left( \bigcup_{x = 0}^{\lfloor \frac{\ell}{2} \rfloor}
      V_{m + d}(S(x)) \right) \cup \left( \bigcup_{x = \ell + 1}^{d}
      V_{m + d}(S(x)) \right) \cup \left( \bigcup_{x = d}^{\lfloor
        \frac{m + \ell}{2} \rfloor} V_{m + d}(S(x)) \right), \]

    and applying Lemmas \ref{I(r) >= 3 bound on S(r)} and \ref{I(r) =
      2 bound on T(r)} and part (ii.b) of Lemma \ref{I(r) = 1 or 2
      bounds on S(r)}, we have

    \[ w(S'') \le \prod_{x = 0}^{\lfloor \frac{\ell}{2} \rfloor}
    w(S(x)) \cdot \prod_{x = \ell + 1}^d w(S(x) \setminus \{ m + x \})
    \cdot \prod_{x = d + 1}^{\min(\lfloor \frac{m + \ell}{2} \rfloor,
      2d)} w(T(x)) \]
    \[ \le \prod_{x = 0}^{\lfloor \frac{\ell}{2} \rfloor} 0.8557 \cdot
    1.618^{|S(x)|} \cdot \prod_{x = \ell + 1}^d 0.8198 \cdot
    1.618^{|S(x)|} \cdot \prod_{x = d + 1}^{\min(\lfloor \frac{m +
        \ell}{2} \rfloor, 2d)} 1.0559 \cdot 1.618^{|T(x)|} \]
    \[ \le 0.8557^{\ell/2} \cdot 0.8198^{d - \ell} \cdot 1.0559^d
    \cdot 1.618^{|S'|} \le 1.618^{|S'|}. \]

    Finally, if $\frac{m + \ell}{2} < d$, then

    \[ S'' = \left( \bigcup_{x = 0}^{\lfloor \frac{\ell}{2} \rfloor}
      V_{m + d}(S(x)) \right) \cup \left( \bigcup_{x = d}^{\lfloor
        \frac{m + \ell}{2} \rfloor} V_{m + d}(S(x)) \right), \]

    and applying Lemma \ref{I(r) >= 3 bound on S(r)} and part (ii.b) of
    Lemma \ref{I(r) = 1 or 2 bounds on S(r)} yields

    \[ w(S'') \le \prod_{x = 0}^{\lfloor \frac{\ell}{2} \rfloor}
    w(S(x)) \cdot \prod_{x = \ell + 1}^{\lfloor \frac{m + \ell}{2}
      \rfloor} w(S(x) \setminus \{ m + x \}) \]
     \[ \le \prod_{x = 0}^{\lfloor \frac{\ell}{2} \rfloor} 0.8557
     \cdot 1.618^{|S(x)|} \cdot \prod_{x = \ell + 1}^{\lfloor \frac{m
         + \ell}{2} \rfloor} 0.8198 \cdot 1.618^{|T(x)|} \]
    \[ \le 1.618^{|S'|}. \]

  \item[] \textbf{Case $I(0) = 2$.} If $d \le \frac{\ell}{2}$, then we
    decompose

    \[ S'' = \left( \bigcup_{x = 0}^{d - 1} V_{m + d}(S(x)) \right)
    \cup \left( \bigcup_{x = d + 1}^{\min(\lfloor \frac{\ell}{2}
        \rfloor, 2d)} V_{m + d}(T(x)) \right) \cup \left( \bigcup_{x =
        \ell + 1}^{\min(\lfloor \frac{m + \ell}{2} \rfloor, \ell + d)}
      V_{m + d}(T(x)) \right). \]

    By Lemmas \ref{I(r) = 1 bound on T(r)} and \ref{I(r) = 2 bound on
      T(r)} and part (ii.b) of Lemma \ref{I(r) = 1 or 2 bounds on
      S(r)}, we find that

    \[ w(S'') \le \prod_{x = 0}^{d - 1} w(S(x) \setminus \{ m + x \})
    \cdot \prod_{x = d + 1}^{\min(\lfloor \frac{\ell}{2} \rfloor, 2d)}
    w(T(x)) \prod_{x = \ell + 1}^{\min(\lfloor \frac{m + \ell}{2}
      \rfloor, \ell + d)} w(T(x)) \]
    \[ \le \prod_{x = 0}^{d - 1} 0.8198 \cdot 1.618^{|S(x)|} \cdot
    \prod_{x = d + 1}^{\min(\lfloor \frac{\ell}{2} \rfloor, 2d)}
    1.0559 \cdot 1.618^{|T(x)|} \prod_{x = \ell + 1}^{\min(\lfloor
      \frac{m + \ell}{2} \rfloor, \ell + d)} 1.1460 \cdot
    1.618^{|T(x)|} \]
    \[ \le 0.8198^d \cdot 1.0559^d \cdot 1.1460^d \cdot 1.618^{|S'|}
    \le 1.618^{|S'|}. \]

    If $\frac{\ell}{2} < d \le \ell$, then

    \[ S'' = \left( \bigcup_{x = 0}^{\ell - d - 1} V_{m + d}(S(x))
    \right) \cup \left( \bigcup_{x = \ell - d}^{\lfloor \frac{\ell}{2}
        \rfloor} V_{m + d}(S(x)) \right) \cup \left( \bigcup_{x = \ell
        + 1}^{\min(\lfloor \frac{m + \ell}{2} \rfloor, \ell + d)} V_{m
        + d}(T(x)) \right), \]

    and 

    \[ w(S'') \le \prod_{x = 0}^{\ell - d - 1} w(S(x) \setminus \{ m +
    x \}) \cdot \prod_{x = \ell - d}^{\lfloor \frac{\ell}{2} \rfloor}
    w(S(x) \setminus \{ m + x , m + \ell - x \} ) \prod_{x = \ell +
      1}^{\min(\lfloor \frac{m + \ell}{2} \rfloor, \ell + d)}
    w(T(x)) \]
    \[ \le \prod_{x = 0}^{\ell - d - 1} 0.8198 \cdot 1.618^{|S(x)|}
    \cdot \prod_{x = \ell - d}^{\lfloor \frac{\ell}{2} \rfloor} 0.5837
    \cdot 1.618^{|S(x)|} \prod_{x = \ell + 1}^{\min(\lfloor \frac{m +
        \ell}{2} \rfloor, \ell + d)} 1.1460 \cdot 1.618^{|T(x)|} \]
    \[ \le 0.8198^{\ell - d} \cdot 0.5837^{d - \ell/2} \cdot 1.1460^d
    \cdot 1.618^{|S'|} \le 1.618^{|S'|}. \]

    If $\ell < d \le \frac{m + \ell}{2}$, then

    \[ S'' = \left( \bigcup_{x = 0}^{\lfloor \frac{\ell}{2} \rfloor}
      V_{m + d}(S(x)) \right) \cup \left( \bigcup_{x = \ell + 1}^{d}
      V_{m + d}(S(x)) \right) \cup \left( \bigcup_{x = d +
        1}^{\min(\lfloor \frac{m + \ell}{2} \rfloor, 2d)} V_{m +
        d}(T(x)) \right), \]
    
    and

    \[ w(S'') \le \prod_{x = 0}^{\lfloor \frac{\ell}{2} \rfloor}
    w(S(x) \setminus \{ m + x, m + \ell - x \}) \prod_{x = \ell +
      1}^{d} w(S(x) \setminus \{ m + x \} ) \prod_{x = d +
      1}^{\min(\lfloor \frac{m + \ell}{2} \rfloor, 2d)} w(T(x)) \]
    \[ \le \prod_{x = 0}^{\lfloor \frac{\ell}{2} \rfloor} 0.5837 \cdot
    1.618^{|S(x)|} \prod_{x = \ell + 1}^{d} 0.7726 \cdot
    1.618^{|S(x)|} \prod_{x = d + 1}^{\min(\lfloor \frac{m + \ell}{2}
      \rfloor, 2d)} 1.1460 \cdot 1.618^{|T(x)|} \]
    \[ \le 0.5837^{\ell / 2} \cdot 0.7726^{d - \ell} \cdot 1.1460^d
    \cdot 1.618^{|S'|} \le 1.618^{|S'|}. \]

    Finally, if $\frac{m + \ell}{2} < d$, then

    \[ S'' = \left( \bigcup_{x = 0}^{\lfloor \frac{\ell}{2} \rfloor}
      V_{m + d}(S(x)) \right) \cup \left( \bigcup_{x = \ell +
        1}^{\lfloor \frac{m + \ell}{2} \rfloor} V_{m + d}(S(x))
    \right), \]

    and

    \[ w(S'') \le \prod_{x = 0}^{\lfloor \frac{\ell}{2} \rfloor}
    w(S(x) \setminus \{ m + x, m + \ell - x \}) \prod_{x = \ell +
      1}^{\lfloor \frac{m + \ell}{2} \rfloor} w(S(x) \setminus \{ m +
    x \}) \]
    \[ \le \prod_{x = 0}^{\lfloor \frac{\ell}{2} \rfloor} 0.5837 \cdot
    1.618^{|S(x)|} \prod_{x = \ell + 1}^{\lfloor \frac{m + \ell}{2}
      \rfloor} 0.7726 \cdot 1.618^{|S(x)|} \]
    \[ \le 0.5837^{\ell / 2} \cdot 0.7726^{\ell / 2} \cdot
    1.618^{|S'|} \le 1.618^{|S'|}. \]

  \item[] \textbf{Case $I(0) = 1$.} Note that if $I(0) = 1$ then $f =
    m + \ell$, and so $d \le \ell$ unless $S''$ is empty, in which
    case $w(S'') \le 1.618^{|S'|}$ holds trivially. If $d \le
    \frac{\ell}{2}$, then

    \[ S'' = \left( \bigcup_{x = 0}^{d} V_{m + d}(S(x)) \right) \cup
    \left( \bigcup_{x = d + 1}^{\min(\lfloor \frac{\ell}{2} \rfloor,
        2d)} V_{m + d}(T(x)) \right). \]

    By part (i.b) of Lemma \ref{I(r) = 1 or 2 bounds on S(r)} and
    Lemma \ref{I(r) = 1 bound on T(r)}, we obtain

    \[ w(S'') \le \prod_{x = 0}^d w(S(x) \setminus \{ m + x \})
    \prod_{x = d + 1}^{\min(\lfloor \frac{\ell}{2} \rfloor, 2d)}
    w(T(x)) \]
    \[ \le \prod_{x = 0}^d 0.7726 \cdot 1.618^{|S(x)|} \prod_{x = d +
      1}^{\min(\lfloor \frac{\ell}{2} \rfloor, 2d)} 1.1460 \cdot
    1.618^{|T(x)|} \]
    \[ \le 0.7726^d \cdot 1.1460^d \cdot 1.618^{|S'|} \le
    1.618^{|S'|}. \]

    If instead $\frac{\ell}{2} < d \le \ell$, then

    \[ S'' = \bigcup_{x = 0}^{\lfloor \frac{\ell}{2} \rfloor} V_{m +
      d}(S(x)). \]

    By part (i.b) of Lemma \ref{I(r) = 1 or 2 bounds on S(r)}, we
    have

    \[ w(S'') \le \prod_{x = 0}^{\lfloor \frac{\ell}{2} \rfloor}
    w(S(x) \setminus \{ m + x \}) \]
    \[ \le \prod_{x = 0}^{\lfloor \frac{\ell}{2} \rfloor} 0.7726 \cdot
    1.618^{|S(x)|} \]
    \[ \le 1.618^{|S'|}. \]

  \end{enumerate}

  This covers all possible values of $I(0)$, establishing that $w(S'')
  \le 1.618^{|S'|}$. We now turn to the relatively simple task of
  bounding $w(S)$ in terms of $w(S'')$. First, note that for any
  admissible subset $U$ of $S$, $U \cup \{ f \}$ is an admissible
  subset of $S \cup \{ f \}$. Furthermore, for any element $x \in E(U
  \cup \{ f \}, S \cup \{ f \})$, we also have $x \in E(U,
  S)$. Finally, if $x \in E'(U, S)$, then clearly $x \in E'(U \cup \{
  f \}, S \cup \{ f \})$. It follows that

  \[ s(U, S) = |E(U, S)| - |E'(U, S)| \]
  \[ \ge |E(U \cup \{ f \}, S \cup \{ f \})| - |E'(U \cup \{ f \}, S
  \cup \{ f \})| \]
  \[ = s(U \cup \{ f \}, S \cup \{ f \}). \]

  \noindent Thus,

  \[ w(S) = \sum_{U \in \mathcal{A}(S)} \varphi^{-s(U, S)} \le \sum_{U
    \in \mathcal{A}(S)} \varphi^{-s(U \cup \{ f \}, S \cup \{ f
    \})} \]
  \[ \le \sum_{U \in \mathcal{A}(S \cup \{ f \})} \varphi^{-s(U, S
    \cup \{ f \})} = w(S \cup \{ f \}). \]

  Using the decomposition $S \cup \{ f \} = S'' \cup S(\frac{\ell}{2})
  \cup S(\frac{m + \ell}{2})$ and Lemma \ref{2r = f mod m bound on
    S(r)}, we find that

  \[ w(S) \le w(S \cup \{ f \}) \le w(S'') w \left( S \left(
      \frac{\ell}{2} \right) \right) w \left( S \left( \frac{m +
        \ell}{2} \right) \right) \]
  \[ \le 1.618^{|S'|} \cdot 1.618^{|S(\frac{\ell}{2})|} \cdot
  1.618^{|S(\frac{m + \ell}{2})|} = 1.618^{f - m + 1} = 1.618^{|S| + d
    + 2}, \]

  \noindent as desired.
\end{proof}

\section{Proof of Lemma \ref{main lemma}}

We are now in a position to prove Lemma \ref{main lemma}. Recall that
$\mathcal{S}(m, f)$ denotes the set of all strongly descended
numerical semigroups having multiplicity $m$ and Frobenius number
$f$. Define $\mathcal{S}(m, f, d)$ to be the subset of $\mathcal{S}(m,
f)$ consisting of those semigroups whose second smallest non-zero
element is $m + d$. For any $\Lambda \in \mathcal{S}(m, f)$, we know
that $f + 1 \in \Lambda$. Hence, $d \le f - m + 1$, and so
$\mathcal{S}(m, f) = \bigcup_{d = 1}^{f - m + 1} \mathcal{S}(m, f,
d)$. We will prove Lemma \ref{main lemma} by decomposing
$\mathcal{S}(m, f)$ in this way and applying Lemma \ref{interval
  weight}.

\begin{proof}[Proof of Lemma \ref{main lemma}]

  As in Lemma \ref{interval weight}, fix values for $m$, $f$, and $d$,
  and define $S = \{ m + d + 1, m + d + 2, \ldots , f - 1 \}$. For any
  $\Lambda \in \mathcal{S}(m, f, d)$, because $\Lambda$ is strongly
  descended, $f + m$ is an effective generator by Lemma
  \ref{Bras-Amoros}. Hence, no two elements of $\Lambda \cap S$ sum to
  $f + m$, and furthermore, if $x \in \Lambda$, then $x + m \in
  \Lambda$, so $\Lambda \cap S$ is an $(m, f, d)$-admissible subset of
  $S$.

  We now give an upper bound on the number of effective generators of
  $\Lambda$ in terms of $E(\Lambda \cap S, S)$ and $E'(\Lambda \cap S,
  S)$. First, if $x \in S$ satisfies $x + m \not \in S$, then $x + m
  \in [f, f + m - 1]$. If in addition we know that $x + m \in
  \Lambda$, then in fact $x + m \in [f + 1, f + m - 1]$, since $f \not
  \in \Lambda$.

  Now, suppose that $x \in E(\Lambda \cap S, S)$. Then, $x - d \in
  \Lambda$, so upon noting that $m + d \in \Lambda$, we find that $x +
  m \in \Lambda$, and $x + m = (x - d) + (m + d)$ is not an effective
  generator. Furthermore, by the definition of $E(\Lambda \cap S, S)$,
  $x \in S$ while $x + m \not \in S$. Thus, $x + m \in [f + 1, f + m -
  1]$.

  For an element $x \in \Lambda \cap S$ that is not in $E'(\Lambda
  \cap S, S)$, we have $x + m \in \Lambda$ but $x + m \not \in \Lambda
  \cap S$, so $x + m \not \in S$. Thus, we again have $x + m \in [f +
  1, f + m - 1]$, and since $x \in \Lambda$ and $m \in \Lambda$, $x +
  m$ is not an effective generator.

  Note that the sets $E(\Lambda \cap S, S)$ and $(\Lambda \cap S)
  \setminus E'(\Lambda \cap S, S)$ are disjoint. For any $x$ in either
  set, $x + m$ is a non-effective generator in the interval $[f + 1, f
  + m - 1]$. It follows that there are at least

  \[ | E(\Lambda \cap S) | + | (\Lambda \cap S) \setminus E'(\Lambda
  \cap S) | \]
  \[ \le | E(\Lambda \cap S) | + | \Lambda \cap S | - | E'(\Lambda
  \cap S) | \]
  \[ = | \Lambda \cap S | - s(\Lambda \cap S) \]

  \noindent elements of $\Lambda$ in the interval $[f + 1, f + m - 1]$
  that are not effective generators. Since the effective generators of
  $\Lambda$ are all in the interval $[f + 1, f + m]$, it follows that

  \[ h(\Lambda) \le m - | \Lambda \cap S | + s(\Lambda \cap S). \]

  \noindent Note that $g(\Lambda) = f - | \Lambda \cap [1, f] | = f -
  | \{ m, m + d \} \cup (\Lambda \cap S) |$, and $|S| = f - m - d -
  1$. Substituting in these identities, we obtain

  \[ h(\Lambda) \le m - (f - 2 - g(\Lambda)) + s(\Lambda \cap S) \]
  \[ g(\Lambda) - h(\Lambda) \ge f - m - 2 - s(\Lambda \cap S) \]
  \[ g(\Lambda) - h(\Lambda) \ge |S| + d - 1 - s(\Lambda \cap S). \]

  \noindent Using Lemma \ref{interval weight}, we find that

  \[ \sum_{\Lambda \in \mathcal{S}(m, f, d)} \varphi^{-g(\Lambda) -
    h(\Lambda)} \le \varphi^{-|S| - d + 1} \sum_{\Lambda \in
    \mathcal{S}(m, f, d)} \varphi^{-s(\Lambda \cap S)} \]
  \[ \le \varphi^{-|S| - d + 1} \sum_{U \in \mathcal{A}(S)}
  \varphi^{-s(U)} = \varphi^{-|S| - d + 1} w(S) \]
  \[ \le \varphi \cdot 1.618^2 \cdot \left( \frac{1.618}{\varphi}
  \right)^{|S| + d} \le 5 \left( \frac{1.618}{\varphi} \right)^{f - m
    - 1}. \]

  \noindent This establishes Lemma \ref{main lemma}.
\end{proof}

\section{Conclusions}

The main result of this paper resolves many of the questions
surrounding the Fibonacci-like behavior of the number of numerical
semigroups of a given genus. However, little has been established
concerning the relationship between $n_g$ and $n_{g + 1}$. In
particular, it remains open whether $n_{g + 2} \ge n_{g + 1} + n_g$,
as conjectured in \cite{BrasAmoros0} and the conjecture that $n_{g +
  1} \ge n_g$ given in \cite{Kaplan} remains unverified for a finite
but large number of $g$.

In addition, we have confirmed Zhao's conjecture in \cite{Zhao} that
the proportion of numerical semigroups $\Lambda$ of a given genus
satsifying $f(\Lambda) < 3m(\Lambda)$ approaches $1$
asymptotically. Thus, in some sense, ``most'' numerical semigroups
satisfy $f < 3m$. It would be interesting to study whether this is
true when counting semigroups by measures of complexity other than
genus. For example, \cite{BlancoPuerto} have considered the problem of
counting the number of numerical semigroups of a given Frobenius
number; one might also ask whether most of these semigroups satisfy $f
< 3m$.

In general, it could be considered whether there is some unified sense
in which one can take the asymptotic limit of semigroups. For any
numerical semigroup $\Lambda$, we have that $g(\Lambda) \le f(\Lambda)
+ 1$, and $f(\Lambda) \le 2 g(\Lambda)$. Thus, we might expect the
sets $\{ \Lambda \mid f(\Lambda) = n \}$ and $\{ \Lambda \mid
g(\Lambda) = n \}$ to behave in similar ways as $n \rightarrow
\infty$. Both genus and Frobenius number can be thought of as proxies
for the ``complexity'' of a numerical semigroup, and it would be
interesting to explore ways to make this precise.

Finally, the proof of Lemma \ref{main lemma} given here (and in
particular the proof of Lemma \ref{interval weight}) is quite
involved. Although the main idea of bounding the weight of an interval
by partitioning it as in Lemma \ref{submultiplicativity} was simple,
computations had to be carried out for many specific cases in order to
obtain sufficiently strong bounds. We hope that by improving upon the
techniques used in this paper, significant simplifications of the
proof are possible.

\section{Acknowledgements}

This work was done as part of the University of Minnesota Duluth REU
program, funded by the National Science Foundation (Award
0754106). The author gratefully acknowledges Nathan Kaplan, who gave
many helpful suggestions for notation and the organization of the
paper. The author would also like to thank Nathan Pflueger and Maria
Monks for advising during the research stage. Finally, the author
would like to thank Joseph Gallian for organizing the Duluth REU
program as well as providing various forms of support and advice along
the way.

\section{Appendix}

We fill in here some of the computational details of the proofs of
Lemmas \ref{I(r) = 1 bound on T(r)} and \ref{I(r) = 2 bound on
  T(r)}. The matrix calculations of this section were done using
Sage.\footnote{www.sagenb.org} Recall that in the proof of Lemma
\ref{I(r) = 1 bound on T(r)} we made the definition

\[ W_n = \left[ \begin{tabular}{c c c}
    $1$ & $1$ & $1$ \\
  \end{tabular} \right] \left[ \begin{tabular}{ c c c }
    $1$ &  $1$ &  $1$ \\
    $\varphi^{-1}$ &  $1$ &  $1$ \\
    $\varphi^{-2}$ &  $1$ &  $\varphi^{-1}$ \\
  \end{tabular} \right]^{n - 1} \left[ \begin{tabular}{ c }
    $1$ \\
    $1$ \\
    $1$ \\
  \end{tabular} \right]. \]

\noindent We now justify in detail the claim that $W_n \le 3 \cdot
1.618^{2n - 2}$. Note that by explicit computation,

\[ W_1 = 3 \]
\[ W_2 = 6 + 2 \varphi^{-1} + \varphi^{-2} \le 7.62 < 3 \cdot 1.618^2 \]
\[ W_3 = 11 + 9 \varphi^{-1} + 6 \varphi^{-2} + \varphi^{-3} \le
19.10 < 3 \cdot 1.618^4. \]

\noindent By the Cayley-Hamilton theorem, we obtain the recurrence

\[ W_{n + 3} = (2.618 \cdots) W_{n + 2} - (0.236 \cdots) W_{n + 1} -
(0.146 \cdots) W_n \]

\noindent for $n \ge 1$. Thus, $W_{n + 1} \le 2.619 W_n < 1.618^2 W_n$
for each $n \ge 3$, and by induction, $W_n \le 3 \cdot 1.618^{2n -
  2}$, as desired.

A similar claim was made in the proof of Lemma \ref{I(r) = 2 bound on
  T(r)}. In that proof, we defined
  
\[ W_n = \left[ \begin{tabular}{c c c c c c}
    $1$ & $1$ & $1$ & $1$ & $1$ & $1$ \\
  \end{tabular} \right] \left[ \begin{tabular}{ c c c c c c }
    $1$ &  $\varphi^{-1}$ &  $1$ &  $\varphi^{-1}$ &  $1$ &  $\varphi^{-2}$ \\
    $1$ &  $1$ &  $1$ &  $1$ &  $1$ &  $1$ \\
    $\varphi^{-1}$ &  $\varphi^{-2}$ &  $1$ &  $\varphi^{-1}$ &  $1$ &  $\varphi^{-3}$ \\
    $\varphi^{-1}$ &  $\varphi^{-1}$ &  $1$ &  $1$ &  $1$ &  $\varphi^{-1}$ \\
    $\varphi^{-1}$ &  $\varphi^{-2}$ &  $\varphi$ &  $1$ &  $\varphi$ &  $\varphi^{-3}$ \\
    $\varphi$ &  $\varphi$ &  $\varphi$ &  $\varphi$ &  $\varphi$ &  $\varphi$ \\
  \end{tabular} \right]^{n - 1} \left[ \begin{tabular}{ c }
    $1$ \\
    $1$ \\
    $1$ \\
    $1$ \\
    $\varphi$ \\
    $\varphi$ \\
  \end{tabular} \right]. \]

\noindent It was claimed that $W_n \le (4 + 2 \varphi) \cdot 1.618^{4n
  - 4}$. As before, we first verify this for small values of $n$.

\[ W_1 = 4 + 2 \varphi \]
\[ W_2 \le 41.51 < (4 + 2 \varphi) \cdot 1.618^4 \]
\[ W_3 \le 226.83 < (4 + 2 \varphi) \cdot 1.618^8 \]
\[ W_4 \le 1225.28 < (4 + 2 \varphi) \cdot 1.618^{12} \]
\[ W_5 \le 6599.87 < (4 + 2 \varphi) \cdot 1.618^{16}. \]

We will prove by induction that $W_{n + 1} \le 6.8 W_n$ for all $n \ge
1$. This can be seen by direct verification for $n \le 4$. Proceeding
inductively, for $n > 4$, we have by the Cayley-Hamilton theorem,

\[ W_{n + 1} = (7.236 \cdots) W_n - (10.708 \cdots) W_{n - 1}
+ (3.965 \cdots) W_{n - 2} - (0.278 \cdots) W_{n - 3} \]
\[ \le 7.237 W_n - (10.707 - 3.966) W_{n - 1} \le \left( 7.237 -
  \frac{10.707 - 3.966}{6.8} \right) W_n \le 6.8 W_n, \]

\noindent where we have used implicitly the basic inequalities $W_k
\ge 0$ and $W_{k + 1} \ge W_k$ for all $k \ge 1$. Noting that $1.618^4
> 6.8$, it is now immediate that $W_n \le (4 + 2 \varphi) \cdot
1.618^{4n - 4}$ for all $n \ge 1$.


\end{document}